\documentclass[a4paper,12pt,oneside,reqno]{amsart}
\usepackage[utf8]{inputenc}
\usepackage{mathtools}

\makeatletter
\@namedef{subjclassname@2020}{\textup{2020} Mathematics Subject Classification}
\makeatother

\allowdisplaybreaks[1]

\usepackage{amsmath,amsthm,amssymb}
\usepackage{hyperref}
\usepackage{amsfonts}
\usepackage{amsmath}
\usepackage{amssymb}
\usepackage{mathrsfs}
\usepackage{tabto}
\usepackage{changepage}
\usepackage{fullpage}
\usepackage{graphicx}
\usepackage{hyperref}
\usepackage{tikz}
\usepackage{caption}   
\usepackage{subfigure}
\usepackage{float}
\usepackage{enumitem}

\usepackage{tikz}
\usetikzlibrary{positioning}
\numberwithin{equation}{section}

\newtheorem{theorem}{Theorem}[section]
\newtheorem{lemma}[theorem]{Lemma}
\newtheorem{corollary}[theorem]{Corollary}
\newtheorem{proposition}[theorem]{Proposition}

\newtheorem{remark}[theorem]{Remark}

\thanks{The research work of the first author is supported by the Institute Research Fellowship of BIT Mesra APO/2024-25/49 and the research work of the second author is supported by ANRF(SERB) research grant TAR/2023/000197}

\title[Independence polynomial]{Structural classification of a graph with Independence number five}
\author[M. Kumari]{Manisha Kumari}
\address{Department of Mathematics, Birla Institute of Technology Mesra
	Ranchi--835 215, India}

\email{phdam10052.24@bitmesra.ac.in}

\author[D. Kumar]{Dinesh Kumar}
\address{Department of Mathematics, Birla Institute of Technology Mesra
	Ranchi--835 215, India}

\email{dineshkumar@bitmesra.ac.in }

\begin{document}
	\begin{abstract}
		The independence polynomial of a simple graph $G$ is given by \( I_G(z) = i_0 + i_1 z + i_2 z^2 + \cdots + i_\alpha z^\alpha \), where \( i_\alpha \) denotes the size of a maximum independent set, also called the independence number of the graph. The independence polynomial has the notable feature of being essentially closed under graph composition (lexicographic product). In this paper, we determine the independence polynomials of size five. For a disconnected graph $G$, we exploit the fact that $I_G(z)$ factors as the product of the independence polynomials of the connected components of $G$. Furthermore, we classify all independence polynomials that can occur for such a disconnected graph $G$ and, by examining their component structures, we characterize the disconnected configurations that may arise.
		
		%	The independence attractor of a graph \( G \), denoted as \( \mathscr{A}(G) \), represents the limiting set of roots of the independence polynomials of repeated lexicographic products of \( G \) with itself, evaluated under the Hausdorff metric. 
		%	This paper focuses on studying an independence attractor.
		%	We present a graph with an independence number of five, whose independence attractors are line segments. Additionally, we determine all possible independence polynomials for such graphs and identify the disconnected configurations that can emerge by analyzing their component structures. 
	\end{abstract}
	\maketitle
	\section{introduction}
	For a simple graph \( G \) which is a graph without loops or parallel edges, an \emph{independent set} is defined as a subset of vertices where no two vertices are adjacent \cite{rosenfeld1964independent}. Let $|V(G)|$ denote the set of vertices of $G$. Denote by \( i_k \) the number of independent sets of cardinality \( k \) in \( G \), and let \(\alpha(G) \) denote the \emph{independence number} of \( G \), which is the maximum size of any independent set. 
	The independence polynomial of \( G \) is defined as 
	\[
	I_G(z) = i_0 + i_1 z + i_2 z^2 + \cdots + i_\alpha z^\alpha,
	\]
	where \( i_0 = 1 \) (since the empty set is the unique independent set of cardinality zero) and \( i_1 = |V(G)| \) (as each individual vertex constitutes a singleton independent set), \cite{hoshino2007independence}. For more details on independence polynomial, one can refer \cite{bollobas1998modern, survey05}.
	The coefficients of the polynomial contain significant combinatorial information: \( i_2 \) counts the number of non-edges in \( G \) 
    % (which is equivalent to counting the edges in the complement graph \( G^c \))
    , while \( i_3 \) denotes the number of independent sets of size three which is related to triangle-free structures in \( G\). The leading coefficient \( i_\alpha \) counts the maximum independent set of size $\alpha$. In this way, the independence polynomial encapsulates the entire series of independent set counts in \( G \) into a single analytic object, allowing for both algebraic and geometric analysis.
	The roots of \( I_G(z) \), known as the independence roots of the graph \( G \), have garnered significant attention. A classical result by Heilmann and Lieb~\cite[Lemma 4.1]{heilmann1972theory} established that if \( G \) is a line graph, then all roots of \( I_G \) are real.
	Independence roots are known to be dense in the complex plane for general graphs~\cite{brown2003independence, chudnovsky2007roots}. However, for specific families, such as paths, they are dense in the real interval \((- \infty, -\frac{1}{4}]\)~\cite{brown2003independence, brown2004location}.
	A decisive new perspective on these roots emerged from their study under iterated graph composition. The lexicographic product (or composition) of graphs \( G \) and \( H \), denoted \( G[H] \), is defined on the vertex set \( V(G) \times V(H) \). In this product, the vertex \((u,v)\) is adjacent to \((u',v')\) if either \( u \) is adjacent to \( u' \) in \( G \), or \( u = u' \) and \( v \) is adjacent to \( v' \) in \( H \). Equivalently, \( G[H] \) is constructed by replacing each vertex of \( G \) with a copy of \( H \) and connecting the vertices across different copies according to the adjacency defined in \( G \). This operation is associative, allowing for the formation of the \( m \)-fold iterated product \( G^m = G[G[\cdots[G]\cdots]] \) $m$-times.
	A remarkable identity, proven by Brown et al. ~\cite[Theorem 1.1]{brown2003independence}, states that 
	\[
	I_{G[H]}(z) = I_G(I_H(z) - 1).
	\]
	By introducing the reduced independence polynomial $f_G(z) = I_G(z) - 1 = i_1 z + i_2 z^2 + \cdots + i_\alpha z^\alpha.$ The above equation can be reformulated as $f_{G[H]} = f_G \circ f_H,$
	which makes the reduced polynomial a natural object for iteration. From this, one can derive, 
	\[
	f_{G^m} = f_G^{\circ m} 
	\]
	% (the \( m \)-th compositional iterate of \( f_G \)), 
    leading to the conclusion that 
	\[
	\mathrm{Roots}(I_{G^m}) = \{ z : f_G^{\circ m}(z) = -1 \} = f_G^{-m}(-1) \quad \text{for all } m \ge 1.
	\]
	The compatibility of functional iteration positions with independence polynomials firmly within the realm of complex analytic dynamics. Brown et al. \cite[Proposition 3.2]{brown2003independence} initiated the study of the independence fractal of a graph \( G \), denoted as $\mathscr{F}(G)$, which is defined as the set of limiting roots of its reduced independence polynomial,
	\begin{equation}\label{eq:fractal}
		\mathscr{F}(G)
		\;=\;
		\lim_{m\to\infty} \mathrm{Roots}\bigl(f_{G^m}\bigr)
		\;=\;
		\lim_{m\to\infty} f_G^{-m}(0).
	\end{equation} Their central theorem establishes that for any graph \( G \neq K_1 \), this limiting set precisely coincides with the Julia set of the reduced independence polynomial, where $K_n$ denotes the complete graph on $n$ vertices.
	For a graph $G$ other than $K_1$,
	$\mathscr{F}(G) \;=\; \mathcal{J}(f_G)$ \cite[Theorem 3.3]{brown2003independence}. The proof rests on the key fact that $0$ is always a repelling fixed
	point of $f_G$, since $f_G'(0)=i_1=|V(G)|>1$, so by the classical
	backward-orbit theorem the iterated preimages $f_G^{-m}(0)$ converge to
	$\mathcal{J}(f_G)$ in the Hausdorff metric. This theorem endows every
	graph with a canonically associated fractal whose geometry encodes deep
	structural properties of $G$. Alongside the independence fractal, one defines the \emph{independence	attractor} of $G$ as
	\begin{equation}\label{eq:attractor}
		\mathscr{A}(G)
		\;=\;
		\lim_{m\to\infty} \mathrm{Roots}\bigl(I_{G^m}\bigr)
		\;=\;
		\lim_{m\to\infty} f_G^{-m}(-1),
	\end{equation}
	where the limit is again taken in the Hausdorff metric \cite[Theorem 3.2.14]{hickman2002roots}. The relationship
	between $\mathscr{A}(G)$ and $\mathscr{F}(G)$ is governed by the
	arithmetic of $-1$ as a root of $I_G$.  Specifically,
	\begin{enumerate}
		\item If $-1$ is not a root of $I_G$, then
		$\mathscr{A}(G)=\mathscr{F}(G)=\mathcal{J}(f_G)$.
		\item If $-1$ is a simple root of $I_G$, then again
		$\mathscr{A}(G)=\mathscr{F}(G)=\mathcal{J}(f_G)$.
		\item If $-1$ is a multiple root of $I_G$, then $-1$ is a
		super-attracting fixed point of $f_G$ lying in the Fatou set,
		and $\mathscr{A}(G)$ is the disjoint union of $\mathscr{F}(G)$
		and $\bigcup_{m\ge 1}f_G^{-m}(-1)$, with $\mathscr{F}(G)$
		serving as the limit set of the latter.
	\end{enumerate}
	Barik et al. \cite{barik2021graphs} rigorously established this precise trichotomy, which was further developed by Khetawat et al. \cite{khetawat2025circles}. 
For a disconnected graph \( G \), represented as the disjoint union of its connected components \( H_1, H_2, \ldots, H_r \), we can express \( G \) as $G = H_1 \cup H_2 \cup \cdots \cup H_r$, then	\begin{equation}\label{prod}
		I_G(z) = \prod_{j=1}^{r} I_{H_j}(z).
	\end{equation}
	This is a direct consequence of the fact that an independent set in \( G \) is simply a choice of an independent set from each component independently. The independence number satisfies the equation 
	$\alpha(G) = \sum_{j=1}^r \alpha(H_j),$
	and the coefficients of the independence polynomial \(I_G(z)\) are the convolution products of the respective coefficient sequences. This multiplicative structure plays a pivotal role in classifying which disconnected graphs possess a given independence polynomial or achieve a specified independence attractor.
	
	The topological classification of independence attractors is a compelling and active area of research. Khetawat et al. \cite[Lemma 3]{khetawat2025circles}, proved that \(K_n^c\) is the only graph whose independence fractal is a circle, thereby ruling out the possibility of any independence attractor being a circle itself. They showed that
	%Denote by $\mathscr{A}(G)$ the independence attractor of a graph $G$.
    if \(\mathscr{A}(G)\) is a line segment, it must be a real interval given by 
	$
	\Bigl[-\tfrac{4}{k}, 0\Bigr]$
	for some \(k \in \{1, 2, 3, 4\}\), \cite[Theorem B]{khetawat2025circles}. The proof shows that the Julia set of \(f_G\) must lie on the real axis, which necessitates that \(f_G\) be conjugate to the Chebyshev polynomial \(T_n\) by means of the affine map 
	$\varphi(z) = \tfrac{k}{2}z + 1.$
	As the Julia set \(\mathcal{J}(T_n)\) is given by \([-1, 1]\), the resulting independence fractal can be expressed as 
	$\varphi^{-1}\bigl([-1, 1]\bigr) = \Bigl[-\frac{4}{k}, 0\Bigr].$
	A crucial combinatorial obstruction regarding the number of vertices and edges of a graph rules out the case \(k = 5\) and confines the parameter \(k\) to \(\{1, 2, 3, 4\}\). Barik et al.\cite{barik2021graphs} established the line segment classification for independence number three, while Khetawat et al.\cite{khetawat2025circles}, extended this classification to all independence numbers, providing explicit examples for independence number four.\\
    This paper expands upon the existing program by focusing on graphs with an independence number of five. We identify all reduced independence polynomials \( f_G(z) \) of degree five for which the Julia set is a line segment. Additionally, we categorize the four possible independence attractors, specifically \(\left[-\frac{4}{k}, 0\right]\) for \(k \in \{1, 2, 3, 4\}\), and provide explicit examples in the situation when the graph is disconnected, corresponding to each attractor. For disconnected graphs with an independence number of five, we utilize the multiplicative structure described in equation~\eqref{prod} to systematically enumerate all admissible component decompositions. The independence polynomial of each component must align with the Chebyshev conjugacy framework, while the convolution constraint on coefficients ensures a finite and computationally manageable classification. 
	
	Our analysis leverages the explicit coefficient formulas established in \cite{khetawat2025circles}, the rational root theorem to eliminate inadmissible parameter values. The relationship between the algebraic rigidity of independence polynomial coefficients and the complex dynamic requirements of the Julia set makes the study quite complicated. We have provided a comprehensive understanding of the case for independence number five.
	
	\medskip
	
	\noindent\textbf{Organization.}\quad  
	This paper is organized as follows: 
	% - Section~\ref{sec:prelim} reviews the necessary background on independence polynomials, lexicographic products, and the Julia-set framework.  
	Section 2 presents the main classification theorem for graphs with an independence number five, specifically for those whose independence attractor is a line segment.  
	 Section 3 addresses the disconnected case, enumerating all admissible component decompositions.  
	Section 4 provides explicit graph constructions that realize each attractor.  

	\section{Independence Polynomial of Graphs with Independence Number Five}\label{sec 2}
	% The first result in this direction is the following:
  In this section, we examine a graph whose independence number is five. In this direction, we recall a result that gives a connection between a polynomial of degree $d\geq 2$ with complex coefficients and the Chebyshev polynomial $T_d$.
    \begin{theorem} \cite{beardon2000iteration}\label{TH1}
     Let $T$ be a polynomial with complex coefficients and degree $d
 \ge 2$. Then the interval $[-1,1]$ is completely invariant under $T$ if and only if $T$ is $\pm T_d$.  
    \end{theorem}
   We mainly focus on the Chebyshev polynomial of degree five. The following theorem describes all possible quintic reduced independence polynomials whose independence fractal is a line segment. 
	\begin{theorem}\label{th2.1}
		The independence fractal of a graph $G$ with independence number five is a line segment if and only if the reduced independence polynomial of the graph is of the form $f_G(z)=25z +50kz^2 +35k^2z^3 + 10k^3z^4 +k^4z^5$ for some $k\in \{1,2,3,4,5\}$.
	\end{theorem}
	\begin{proof}
		Consider a graph $G$ with independence number five, whose reduced independence polynomial is
		\[ f_G(z)= i_1z +i_2z^2+i_3z^3 +i_4z^4 +i_5z^5. \]
		By  \cite[Theorem 1.1]{barik2021graphs} if the independence fractal $\mathscr{F}(G)$ is a line segment $M$, then the Julia set of $f_G$ coincides with $M$. From \cite[Theorem 3.2.4]{beardon2000iteration} (for any rational map $R$ having degree at least two, the Fatou set and the Julia set are completely invariant under $R$), this line segment $M$ is completely invariant under $f_G$, i.e.,
		\[ f_G(M) \subseteq M \quad \text{and} \quad f_G^{-1}(M)\subseteq M. \]
		% Let $w_0$, $l$, and $\psi$ denote, respectively, the midpoint, the length, and the angle of inclination of $M$ with respect to the positive real axis. Then there exists a linear map $\phi(z)=az+b$ that carries the interval $[-1,1]$ onto the segment $M$ where $a=\frac{le^{i \psi}}{2}, b=w_0$. Consequently, we have
		% \[
		% f_G = \phi \circ T_5(z) \circ  \phi^{-1},
		% \]
  %       is a quintic polynomial that leaves the interval $[-1,1]$ completely invariant. By Theorem \ref{TH1}, we then have $\phi^{-1}(f_G(\phi)) = \pm T_5$. We now determine the coefficients of $f_G$. Set $w = az + b$ and note that $T_5(z) = 16z^5 - 20z^3 + 5z$. Then,
  Let $w_0, l$ and $\psi$ denote, respectively, the midpoint, length, and angle of inclination of segment M relative to the positive real axis. When $a = \frac{le^{i \psi}}{2}$, and $b = w_0$, there is a linear map defined by $\phi(z) = az + b$, which converts the interval $[-1, 1]$ onto the segment $M$. As a result, we have $f_G = \phi \circ T_5(z) \circ \phi^{-1}$ which describes a quintic polynomial that maintains complete invariance of the interval $[-1, 1]$. We can infer from Theorem \ref{TH1} that $\phi^{-1}(f_G(\phi)) = \pm T_5$. The coefficients of $f_G$ will now be computed. Remember that $T_5(z) = 16z^5 - 20z^3 + 5z$, and let $w = az + b$. Then,
		\begin{align} \label{Eq1}
			f_G(w)
			&= a\Big[16\Big(\frac{w-b}{a}\Big)^5 - 20\Big(\frac{w-b}{a}\Big)^3 + 5\Big(\frac{w-b}{a}\Big)\Big] + b \notag\\
			&= \frac{16}{a^4}w^5 - \frac{80b}{a^4}w^4 + \Big(\frac{160b^2}{a^4}- \frac{20}{a^2}\Big)w^3 + \Big(\frac{-160b^3}{a^4}+\frac{60b}{a^2}\Big)w^2 \notag \\
			&\quad + \Big(\frac{80b^4}{a^4}-\frac{60b^2}{a^2}+5\Big)w +
			\Big(\frac{-16b^5}{a^4}+\frac{20b^3}{a^2}-4b\Big). 
		\end{align}
		As there is no constant term in the reduced independence polynomial, the final coefficient must equal zero. Therefore, we have the equation: \[ \frac{-16b^5}{a^4} + \frac{20b^3}{a^2} - 4b = 0. \] The solutions to this equation for \( b \) include \( 0, \pm a, \) and \( \pm \frac{a}{2} \). We will begin with the case \( b = 0 \). Substituting \( b = 0 \) into Equation \eqref{Eq1} results in \( i_2 = 0 \). Since \( i_k = 0 \) implies \( i_j = 0 \) for all \( j \ge k \), it follows that \( i_3 = i_4 = i_5 = 0 \). This contradicts the condition that \( i_5 \neq 0 \). Thus, we can conclude that \( b = 0 \) is not a valid solution.
% For the cases $b = a,\ a/2,\ -a/2,$ 
%  at least one of the coefficients $i_1,\dots,i_5$ becomes negative (for instance,
% $b = a \Rightarrow i_4 < 0,\quad
% b = a/2 \Rightarrow i_4 = -\frac{40}{a^3} < 0,\quad
% b = -a/2 \Rightarrow i_1 = 5 - 15 + 5 = -5 < 0.$)
% Since all coefficients must be non-negative, these three cases are also ruled out.
For the cases where \( b = a \), \( b = \frac{a}{2} \), and \( b = -\frac{a}{2} \), at least one of the coefficients \( i_1, i_2, i_3, i_4, i_5 \) becomes negative. For instance
$b = a \Rightarrow i_4 < 0,\quad b = a/2 \Rightarrow i_4 = -\frac{40}{a^3} < 0,\quad
 b = -a/2 \Rightarrow i_1 = 5 - 15 + 5 = -5 < 0.$
Since all coefficients must be non-negative, these three cases are ruled out as valid options.
Finally in the case $b = -a$,
 all coefficients $i_1,\dots,i_5$ are positive. Plugging in $b=-a$ in \eqref{Eq1} yields $i_1 = 25$. The image  $\phi[-1,1]$ is $[-2a,0]$, which coincides with the known line segment attractor $[-4/k,0]$ provided $2a = 4/k$, i.e, $a = 2/k$, establishing the required reduced independence polynomial.

	According to \cite[Lemma 5]{manna2025connectedness} and \cite[Remark 3.2]{khetawat2025circles}, the reduced independence polynomials of graphs with an independence number of five, whose independence attractor is a line segment, are given by $f_G(z) = 25z + 50kz^2 + 35k^2z^3 + 10k^3z^4 + k^4z^5$. The corresponding independence polynomial is
	\begin{equation} \label{Eq2}
^kI_G(z)=I_G(z)=1+25z+50kz^2+35k^2z^3+10k^3z^4+k^4z^5, 
	\end{equation}
%     for some positive integer $k$. Since $i_1(G)=25$, $G$ has $25$ vertices. Any $j$-independent set is a $j$-subset of $V(G)$, so $i_j(G)\le\binom{25}{j}$ for all $j$. For $j=2, \, 50k \le \binom{25}{2} = 300 \implies k \le 6.$ If $k=6$, then $i_2(G) = 300 = \binom{25}{2}$, so every pair of vertices is independent and $G$ is edgeless. Then every $3$-subset is independent, giving $i_3(G) = \binom{25}{3} = 2300,$ but, the polynomial gives $i_3(G) = 35\cdot 6^2 = 1260,$
% a contradiction. Thus, $k \neq 6$ and since $k$ is a positive integer, we must have $k \in \{1,2,3,4,5\}$.
for some positive integer $k$. As $i_1=25$, $G$ has $25$ vertices. Since $i_2\leq \binom{25}{2}$, this gives $ 50k \le 300$ which implies that $k \le 6.$ If \( k = 6 \), then \( i_2 = 300 = \binom{25}{2} \), indicating that every pair of vertices is independent, meaning that \( G \) has no edges. In this case, $i_3$ equals $ \binom{25}{3} = 2300 $, while the polynomial yields \( i_3 = 35 \cdot 6^2 = 1260 \). This is a contradiction. Thus, $k \neq 6$ and since $k$ is a positive integer, we must have $k \in \{1,2,3,4,5\}$.\\
% matching the restriction found by Barik, Nayak, and Pradhan (2021) for the cubic case with independence number $3$.\\
For the backward implication, if \( f_G(z) = {}^k f_G(z) \) for some \( k \in \{1, 2, 3, 4,5\} \), then the map \( \phi_k(z) = \frac{2}{k}(z - 1) \) conjugates \( {}^k f_G(z) \) to \( T_5(z) = 16z^5 - 20z^3 + 5z \). More precisely, we have \( T_5(z) = \phi_k^{-1} \circ f_5 \circ \phi_k(z) \) for every \( z \) and for each \( k \in \{1, 2, 3, 4,5\} \).
The result now follows from Theorem \ref{TH1}.
    \end{proof}
We begin by searching for disconnected graphs whose independence polynomial is $^kI_G(z)$. 
	It is known that if $G \cup H$ denotes the disjoint union of two graphs $G$ and $H$, then $I_{G \cup H}(z) = I_G(z) I_H(z)$. More generally, the independence polynomial of a disconnected graph is equal to the product of the independence polynomials of its connected components \cite[Theorem 3.0.12]{hickman2002roots}. 
	
	Our procedure is as follows. First, we determine all possible factorizations of the polynomials $^kI_G(z)$ into factors
	with positive integer coefficients. Each such factor is a possible independence polynomial of
	a component graph. Next, we verify, for each such factor, whether there
	exists a graph whose independence polynomial equals that factor.\\
    The following result gives a connection between the independence attractor and the independence fractal of a graph $G$ with independence number five. 
	\begin{corollary}
		The independence attractor and the independence fractal of a graph $G$ with independence number five coincide if the independence fractal is a line segment.
	\end{corollary}
	\begin{proof}
	From Equation \eqref{Eq2}, one can easily see that $-1$ is a root of $I_G(z)$ for all $k\in \{1,2,3,4,5\}$. A direct calculation further shows that $I'_G(-1) \neq 0$ for each $k\in \{1,2,3,4,5\}$, and hence $-1$ is a simple root of $I_G(z)$. The claim then follows from Theorem \ref{th2.1} together with \cite[Remark 2]{brown2003independence}, which states that for any graph $G$ whose independence number is five, the independence fractal coincides with the independence attractor whenever the latter is a line segment.
\end{proof}
	\section{Different Components of The Independence Polynomial} \label{sec 3}
   In this section, we examine the possible number of connected components in a disconnected graph whose independence polynomial is given by $^kI_G(z)=I_G(z)=1+25z+50kz^2+35k^2z^3+10k^3z^4+k^4z^5$. We have established that this number can be at most three. The following result makes this statement precise.
	% A disconnected graph $G$ with independence number five can have either two or three connected components. The following proposition makes this statement precise.
	%	\section{no disconnected graph of four and five components}
	\begin{proposition}\label{prop 1}
		If $G$ is a disconnected graph with independence number five and $\mathscr{A}(G)$ is a line segment, then $G$ has at most three connected components.
		\begin{enumerate}
			\item  When the graph has three components, its independence polynomial has the decomposition $(1+z)(1+12z+16z^2)(1+12z+16z^2)$.
			\item When the graph has two components, its independence polynomial has one of the following decompositions: $(1+12z+16z^2)(1+13z+28z^2+16z^3); (1+z)(1+24z+26z^2+9z^3+z^4); (1+z)(1+24z+76z^2+64z^3+16z^4); (1+z)(1+24z+126z^2+189z^3+81z^4)$; $(1+z)(1+24z+176z^2+384z^3+256z^4)$.
		\end{enumerate}
      Furthermore, $\mathscr{A}(G)$ is $[-1,0]$ if the independence polynomial decomposition is one of the following: $(1+z)(1+12z+16z^2)(1+12z+16z^2), (1+12z+16z^2)(1+13z+28z^2+16z^3), (1+z)(1+24z+176z^2+384z^3+256z^4)$; $\mathscr{A}(G)$ is $[-2,0]$ if the independence polynomial has the decomposition $(1+z)(1+24z+76z^2+64z^3+16z^4)$; $\mathscr{A}(G)$ is $[-4/3,0]$ if the independence polynomial has the decomposition $(1+z)(1+24z+126z^2+189z^3+81z^4)$; and $\mathscr{A}(G)$ is $[-4,0]$ if the independence polynomial has the decomposition $(1+z)(1+24z+26z^2+9z^3+z^4)$.
	\end{proposition}
	\begin{proof}
Suppose $G$ consists of five components $H_i$, each having $v_i$ vertices for $i = 1,2,3,4,5$. Then the independence polynomial of every $H_i$ has to be linear,
\begin{align*}
			^kI_G(z) 
			&= (1+v_1z)(1+v_2z)(1+v_3z)(1 +v_4z)(1+v_5z)\\
			&=1+(v_1+v_2+v_3+v_4+v_5)z\\
			&\quad +(v_1v_2+v_1v_3+v_1v_4+v_1v_5+v_2v_3+v_2v_4+v_2v_5+v_3v_4+v_3v_5+v_4v_5)z^2 \\
			&\quad+(v_1v_2v_3+v_1v_2v_4+v_1v_2v_5+v_1v_3v_4+v_1v_3v_5\\
			&\quad +v_1v_4v_5+v_2v_3v_4+v_2v_3v_5+v_3v_4v_5+v_2v_4v_5)z^3\\
			&\quad + (v_1v_2v_3v_4+v_1v_2v_3v_5+v_2v_3v_4v_5+v_1v_3v_4v_5+v_1v_2v_4v_5)z^4+(v_1v_2v_3v_4v_5)z^5.
		\end{align*}
		Since $^kI_G(z) =1+25z+50kz^2+35k^2z^3+10k^3z^4+k^4z^5$, any admissible 5-tuple\\ $(v_1, v_2, v_3, v_4, v_5)$ must satisfy the following system of equations
		\begin{align}
			&v_1 + v_2 + v_3 + v_4 + v_5 = 25,\\
			&v_1v_2 + v_1v_3 + v_1v_4 + v_1v_5 + v_2v_3 + v_2v_4 + v_2v_5 + v_3v_4 + v_3v_5 + v_4v_5 = 50k,\\
			&v_1v_2v_3 + v_1v_2v_4 + v_1v_2v_5 + v_1v_3v_4 + v_1v_3v_5 + v_1v_4v_5 + v_2v_3v_4 \notag\\
			&+ v_2v_3v_5 + v_2v_4v_5 + v_3v_4v_5 = 35k^2,\\
			&v_1v_2v_3v_4 + v_1v_2v_3v_5 + v_1v_2v_4v_5 + v_1v_3v_4v_5 + v_2v_3v_4v_5 = 10k^3,\\
			&v_1v_2v_3v_4v_5 = k^4,
		\end{align}
		for some $k \in \{1,2,3,4,5\}$. It is evident that for $k = 1,2,3,4,5$, there is no 5-tuple\\ $(v_1, v_2, v_3, v_4, v_5)$ that simultaneously satisfies
		$v_1 + v_2 + v_3 + v_4 + v_5 = 25
		$
		and
		$v_1v_2v_3v_4v_5 = k^4.$\\
		Hence, the graph $G$ cannot have five components.
		Suppose $G$ consists of four components $H_i$ with $v_i$ vertices for $i = 1,2,3,4$, three of them ($H_1, H_2$, and $H_3$) have independence number one, whereas $H_4$ has independence number two. Then, 
		\begin{align*}
			^kI_G(z)
			&= (1+v_1z)(1 +v_2z)(1+v_3z)(1 +v_4z +mz^2)\\
			&=1+(v_1+v_2+v_3+v_4)z+(v_1v_2+v_1v_3+v_1v_4+v_2v_3+v_2v_4+v_3v_4+m)z^2\\
			&\quad +(v_1v_2v_3+v_1v_2v_4+v_2v_3v_4+v_1v_3v_4+v_1m+v_2m+v_3m)z^3\\
			&\quad+(v_1v_2v_3v_4+v_1v_2m+v_1v_3m+v_2v_3m)z^4+(v_1v_2v_3m)z^5.
		\end{align*}
		%	For the remaining values of $k$, we first list all the values of $v_1,v_2,v_3,m$ satisfying
		%	$v_1v_2v_3m = k^4$, discarding the cases where any of $v_1,v_2$ or $v_3$ is greater than or equal
		%	to 25. Then we find $v_4$ using $v_1 + v_2 + v_3 +v_4 = 25$ and finally we check if the
		%	other equations are satisfied or not. This is done for k = 2,3,4 in Tables 1, 2, 3
		%	respectively.
		%	For k = 2, $n1n2n3 +mn1 +mn2 = 32$. But it is clear from Table 1 that this is
		%	not possible.
		%	
		Each possible combination of $(m, v_1,v_2,v_3,v_4)$ must satisfy
		\begin{align}
			&v_1+v_2+v_3+v_4=25\\ 	
	&v_1v_2+v_1v_3+v_1v_4+v_2v_3+v_2v_4+v_3v_4+m=50k,\\
&v_1v_2v_3+v_1v_2v_4+v_2v_3v_4+v_1v_3v_4+v_1m+v_2m+v_3m=35k^2,\\
			&v_1v_2v_3v_4+v_1v_2m+v_1v_3m+v_2v_3m=10k^3,\\
			&v_1v_2v_3m=k^4,
		\end{align}
		%	\begin{equation}
			%&	v_1+v_2+v_3+v_4=25\\
			%&	v_1v_2+v_1v_3+v_1v_4+v_2v_3+v_2v_4+v_3v_4+m=50k,\\
			%&	v_1v_2v_3+v_1v_2v_4+v_2v_3v_4+v_1v_3v_4+v_1m+v_2m+v_3m=35k^2,\\
			%&	v_1v_2v_3v_4+v_1v_2m+v_1v_3m+v_2v_3m=10k^3,\\
			%&	v_1v_2v_3m=k^4
			%	\end{equation}
		for some $k \in \{1,2,3,4,5\}$.  
		For $k=1$, we must have $(v_1,v_2,v_3,m)=(1,1,1,1)$ since $v_1v_2v_3m=1$. This gives $v_4=22$. Also, $
		v_1v_2+v_1v_3+v_1v_4+v_2v_3+v_2v_4+v_3v_4+m=70,
		$ whereas the second condition requires this sum to be $50$, so $k=1$ is impossible. To handle the remaining values of $k$, we first enumerate all quadruples $(v_1,v_2,v_3,m)$ such that $v_1v_2v_3m=k^4$. We discard any cases where at least one of $v_1,v_2,$ or $v_3$ is $\geq 25$.
		For each remaining case, we solve $v_1+v_2+v_3+v_4=25$ to obtain $v_4$ and then check whether the other equations are satisfied. Tables (\ref{t1}- \ref{T3}) summarize the outcomes for $k=2,3,4$ and $5$, respectively. \\
		For $k=2$, the second condition becomes $v_1v_2+v_1v_3+v_1v_4+v_2v_3+v_2v_4+v_3v_4+m=100.$
		Table \ref{t1} shows that no choice of $(v_1,v_2,v_3,m,v_4)$ yields this value, so $k=2$ does not produce a valid solution.
		%\begin{center}
		%	\begin{tabular}{|c|c|c|c|c|c|}
			%		\hline
			%		$m$ & $v_1$ & $v_2$ & $v_3$ & $v_4=25-(v_1+v_2+v_3)$ &
			%		$\begin{aligned}
				%			&v_1v_2+v_1v_3+v_1v_4+v_2v_3\\
				%			&+v_2v_4+v_3v_4+m
				%		\end{aligned}$ \\ \hline
			%			16 & 1 & 1 & 1& 22& 85 \\ \hline
			%		1 & 16 & 1 & 1& 7 & 160 \\ \hline
			%		1 & 1 & 16 &1 &7 &  160\\ \hline
			%		1 &1  &1  &16 &7 & 160 \\ \hline
			%		8&2  &1  &1 &21 & 97 \\ \hline
			%		8& 1 &2  &1 &21 & 97 \\ \hline
			%		8&1  &1  &2 &21 & 97 \\ \hline
			%		1&8  &2  &1 &14 & 181 \\ \hline
			%		1& 1 & 2 &8 & 14& 181  \\ \hline
			%		2& 8 & 1 &1 & 15& 169  \\ \hline
			%		2& 1 & 8 &1 & 15& 169  \\ \hline
			%		2& 1 & 1 &8 & 15&  169 \\ \hline
			%		4& 4 &1  &1 & 19& 127 \\ \hline
			%		4&1  &4  &1 & 19& 127 \\ \hline
			%		4&1  &1  &4 & 19& 127 \\ \hline
			%		1 &4  &4&1 &16 & 169 \\ \hline
			%		1&1  &4  &4 &16 & 169 \\ \hline
			%		1&4  &1  &4 &16 & 169 \\ \hline
			%		4 &2  &2  &1 & 20& 112 \\ \hline
			%		4&2  &1  &2 & 20& 112 \\ \hline
			%		4&1  &2  &2 & 20& 112 \\ \hline
			%		2&1  &4  &2 & 18& 142  \\ \hline
			%		2&1  &2  &4 & 18& 142 \\ \hline
			%		2&4  &1  &2 & 18& 142  \\ \hline
			%		2&4  &2  &1 & 18& 142  \\ \hline
			%		2&2  &4  &1 & 18& 142 \\ \hline
			%		2&2  &1  &4 & 18& 142  \\ \hline
			%		2&2 &2&2& 19& 128 \\ \hline
			%	\end{tabular}
		%\end{center}
		\begin{table}
			\centering
			\caption{Possible values of $v_1,v_2,v_3,v_4$ and $m$ for $k=2: v_1v_2+v_1v_3+v_1v_4+v_2v_3+v_2v_4+v_3v_4+m=100.$}
			\label{t1}
			\scriptsize
			\begin{tabular}{|c|c|c|c|c|c|c|}
				\hline
				$m$ & $v_1$ & $v_2$ & $v_3$ & $v_4=25-(v_1+v_2+v_3)$ & \begin{tabular}{c}
					$v_1v_2+v_1v_3+v_1v_4+v_2v_3$\\$+v_2v_4+v_3v_4+m$
				\end{tabular} \\ \hline
				16 & 1 & 1 & 1& 22& 85 \\ \hline
				1 & 16 & 1 & 1& 7 & 160 \\ \hline
				1 & 1 & 16 &1 &7 &  160 \\ \hline
				1 &1  &1  &16 &7 & 160\\ \hline
				8&2  &1  &1 &21 & 97 \\ \hline
				8& 1 &2  &1 &21 & 97  \\ \hline
				8&1  &1  &2 &21 & 97 \\ \hline
				1&8  &2  &1 &14 & 181  \\ \hline
				1& 1 & 2 &8 & 14& 181  \\ \hline
				2& 8 & 1 &1 & 15& 169 \\ \hline
				2& 1 & 8 &1 & 15& 169  \\ \hline
				2& 1 & 1 &8 & 15&  169\\ \hline
				4& 4 &1  &1 & 19& 127 \\ \hline
				4&1  &4  &1 & 19& 127 \\ \hline
				4&1  &1  &4 & 19& 127\\ \hline
				1 &4  &4&1 &16 & 169\\ \hline
				1&1  &4  &4 &16 & 169\\ \hline
				1&4  &1  &4 &16 & 169\\ \hline
				4 &2  &2  &1 & 20& 112\\ \hline
				4&2  &1  &2 & 20& 112 \\ \hline
				4&1  &2  &2 & 20& 112  \\ \hline
				2&1  &4  &2 & 18& 142 \\ \hline
				2&1  &2  &4 & 18& 142 \\ \hline
				2&4  &1  &2 & 18& 142 \\ \hline
				2&4  &2  &1 & 18& 142  \\ \hline
				2&2  &4  &1 & 18& 142 \\ \hline
				2&2  &1  &4 & 18& 142  \\ \hline
				2&2 &2&2& 19& 128  \\ \hline
			\end{tabular}
		\end{table}
		For $k=3$, we have $v_1v_2+v_1v_3+v_1v_4+v_2v_3+v_2v_4+v_3v_4+m=150$ and $v_1v_2v_3+v_1v_2v_4+v_2v_3v_4+v_1v_3v_4+v_1m+v_2m+v_3m=315.$ As indicated in Table \ref{t2}, these conditions cannot be satisfied.
		\begin{table}[h]
			\centering
			\caption{Possible values of $v_1,v_2,v_3,v_4$ and $m$ for $k=3: v_1v_2+v_1v_3+v_1v_4+v_2v_3+v_2v_4+v_3v_4+m=150.$ }
			\label{t2}
			\scriptsize
			\begin{tabular}{|c|c|c|c|c|c|c|c|}
				\hline
				$m$ & $v_1$ & $v_2$ & $v_3$ &
				\begin{tabular}{c}
					$v_4=25-$ \\ $(v_1+v_2+v_3)$ 
				\end{tabular}& \begin{tabular}{c}
					$v_1v_2+v_1v_3$ \\ $+v_1v_4+v_2v_3$ \\ $+v_2v_4+v_3v_4+m$
				\end{tabular}&
				\begin{tabular}{c}
					$v_1v_2v_3+v_1v_2v_4$ \\ $+v_2v_3v_4+v_1v_3v_4$ \\ $+v_1m+v_2m+v_3m$
				\end{tabular}\\ \hline
				81& 1&1&1&22&150&310 \\ \hline
				% 1& 81&1&1&22&150&310& Not possible \\ \hline
				% 1& 1&81&1&22&150&310& Not possible \\ \hline
				% 1& 1&1&81&22&150&310& Not possible \\ \hline
				27& 3&1&1&20&134&-- \\ \hline
				27& 1&3&1&20&13&-- \\ \hline
				27& 1&1&3&20&134&-- \\ \hline
				9& 9&1&1&14&182&-- \\ \hline
				9& 1&9&1&14&182&-- \\ \hline
				9&1 &1&9&14&182&-- \\ \hline
				1& 9&9&1&6&160&-- \\ \hline
				1& 9&1&9&6&160&--\\ \hline
				1& 1&9&9&6&160&--\\ \hline	
				9& 3&3&1&18&150&342 \\ \hline
				9& 3&1&3&18&150&342 \\ \hline
				9& 1&3&3&18&150&342 \\ \hline
				3& 1&3&9&12&195&-- \\ \hline
				3& 1&9&3&12&195&--\\ \hline
				3& 9&3&1&12&195&--\\ \hline
				3& 9&1&3&12&195&-- \\ \hline	
				3& 1&3&9&12&195&-- \\ \hline
				3& 3&1&9&12&195&-- \\ \hline
				3& 3&9&1&12&195&-- \\ \hline		
			\end{tabular}
		\end{table}
		For $k=4,$ the equation $v_1v_2+v_1v_3+v_1v_4+v_2v_3+v_2v_4+v_3v_4+m=200$ and $v_1v_2v_3+v_1v_2v_4+v_2v_3v_4+v_1v_3v_4+v_1m+v_2m+v_3m$ must hold. However, Table \ref{t3} indicates that no such configuration exists.
		\begin{table}
			\centering
			\caption{Possible values of $v_1,v_2,v_3,v_4$ and $m$ for $k=4: v_1v_2+v_1v_3+v_1v_4+v_2v_3+v_2v_4+v_3v_4+m=200$.}
			\label{t3}
			\scriptsize
			\begin{tabular}{|c|c|c|c|c|c|c|c|}
				\hline
				$m$ & $v_1$ & $v_2$ & $v_3$ &
				\begin{tabular}{c}
					$v_4=25-$ \\ $(v_1+v_2+v_3)$ 
				\end{tabular}& \begin{tabular}{c}
					$v_1v_2+v_1v_3$ \\ $+v_1v_4+v_2v_3$ \\ $+v_2v_4+v_3v_4+m$
				\end{tabular}&
				\begin{tabular}{c}
					$v_1v_2v_3+v_1v_2v_4$\\$+v_2v_3v_4+v_1v_3v_4$\\$+v_1m+v_2m+v_3m$
				\end{tabular}\\	\hline
				256&1&1&1&22&325 &--\\ \hline
				128&2&1&1&21&217&--\\ \hline
				128&1&2&1&21&217&--\\ \hline
				128&1&1&2&21&217&--\\ \hline
				64&4&1&1&19&187&--\\ \hline
				64&1&4&1&19&187&--\\ \hline
				64&1&1&4&19&187&--\\ \hline
				64&2&2&1&20&172&--\\ \hline
				64&1&2&2&20&172&--\\ \hline
				64&2&1&2&20&172&--\\ \hline
				32&8&1&1&15&199&--\\ \hline
				32&1&8&1&15&199&--\\ \hline
				32&1&1&8&15&199&--\\ \hline
				32&2&4&1&18&172&--\\ \hline
				32&2&1&4&18&172&--\\ \hline
				32&1&4&2&18&172&--\\ \hline
				32&4&2&1&18&172&--\\ \hline
				32&4&1&2&18&172&--\\ \hline
				32&2&2&2&19&158&--\\ \hline
				16&16&1&1&7&175&--\\ \hline
				16&1&16&1&7&175&--\\ \hline
				16&1&1&16&7&175&--\\ \hline
				16&8&2&1&14&196&--\\ \hline
				16&8&1&2&14&196&--\\ \hline
				16&2&8&1&14&196&--\\ \hline
				16&2&1&8&14&196&--\\ \hline
				16&1&2&8&14&196&--\\ \hline
				16&1&8&2&14&196&--\\ \hline
				16&4&4&1&16&184&--\\ \hline
				16&4&1&4&16&184&--\\ \hline
				16&1&4&4&16&184&--\\ \hline
				16&4&2&2&17&172&--\\ \hline
				16&2&4&2&17&172&--\\ \hline
				16&2&2&4&17&172&--\\ \hline
				8&8&2&2&13&200&596\\ \hline
				8&2&8&2&13&200&596\\ \hline
				8&2&2&8&13&200&596\\ \hline
				2&8&8&2&7&224&--\\ \hline
				2&2&8&8&7&224&--\\ \hline 
				8&4&4&2&15&190&--\\ \hline
				8&4&2&4&15&190&--\\ \hline
				8&4&4&2&15&190&--\\ \hline
				4&4&2&8&11&214&--\\ \hline
				4&4&8&2&11&214&--\\ \hline
				4&8&2&4&11&214&--\\ \hline
				4&8&4&2&11&214&--\\ \hline
				4&2&4&8&11&214&--\\ \hline
				4&2&8&4&11&214&--\\ \hline
				4&4&4&4&13&208&--\\ \hline
			\end{tabular}
		\end{table}
        For $k = 5$, the relations $v_1v_2+v_1v_3+v_1v_4+v_2v_3+v_2v_4+v_3v_4+m=250$ and $v_1v_2v_3+v_1v_2v_4+v_2v_3v_4+v_1v_3v_4+v_1m+v_2m+v_3m$ are required to hold. Nevertheless, Table \ref{T3} shows that no configuration satisfying these conditions exists.
        \begin{table}
			\centering
			\caption{Possible values of $v_1,v_2,v_3,v_4$ and $m$ for $k=5: v_1v_2+v_1v_3+v_1v_4+v_2v_3+v_2v_4+v_3v_4+m=250$.}
			\label{T3}
			\scriptsize
			\begin{tabular}{|c|c|c|c|c|c|c|c|}
				\hline
				$m$ & $v_1$ & $v_2$ & $v_3$ &
				\begin{tabular}{c}
					$v_4=25-$ \\ $(v_1+v_2+v_3)$ 
				\end{tabular}& \begin{tabular}{c}
					$v_1v_2+v_1v_3$ \\ $+v_1v_4+v_2v_3$ \\ $+v_2v_4+v_3v_4+m$
				\end{tabular}&
				\begin{tabular}{c}
					$v_1v_2v_3+v_1v_2v_4$\\$+v_2v_3v_4+v_1v_3v_4$\\$+v_1m+v_2m+v_3m$
				\end{tabular}\\	\hline
625&1&1&1&22&694&--\\ \hline
125&5&1&1&18&262&--\\ \hline
125&1&5&1&18&262&--\\ \hline
125&1&1&5&18&262&--\\ \hline
25&5&5&1&14&214&--\\ \hline
25&5&1&5&14&214&--\\ \hline
25&1&5&5&14&214&-- \\ \hline
5&5&5&5&10&230&--\\ \hline
                \end{tabular}
		\end{table}
		Hence, in none of the five cases can $^kI_G(z)$ be decomposed into four separate components.\\
		%\item
		We now suppose $G$ consists of three components $H_i$ with $v_i$ vertices for $i = 1,2,3$. Here, two components have independence number two, and the remaining one has independence number one. Assume 
		$^kI_{H_1}(z) = 1 + v_1 z$, $^kI_{H_2}(z) = 1 + v_2 z + m_1 z^2$, and $^kI_{H_3}(z) = 1 + v_3 z + m_2 z^2$ for some natural numbers $m_1,m_2$. Then,
		\begin{align*}
			^kI_G(z) 
			&= (1+v_1z)(1+v_2z+m_1z^2)(1+v_3z+m_2z^2)\\
			&= 1+(v_1+v_2+v_3)z+(m_1+m_2+v_1v_2+v_1v_3+v_2v_3)z^2\\ &\quad+(v_1m_2+v_2m_2+v_3m_1+v_1v_2v_3+v_1m_1)z^3+(m_1m_2+v_1v_2m_2+v_1v_3m_1)z^4\\
			&\quad+ v_1m_1m_2z^5.
		\end{align*}
		Every feasible value of $(m_1,m_2,v_1,v_2,v_3)$ must fulfill these five given equations
		\begin{align}
			&v_1+v_2+v_3=25,\\
			& m_1+m_2+v_1v_2+v_1v_3+v_2v_3=50k,\\
			&v_1m_2+v_2m_2+v_3m_1+v_1v_2v_3+v_1m_1=35k^2,\\
			& m_1m_2+v_1v_2m_2+v_1v_3m_1=10k^3,\\
			&v_1m_1m_2=k^4,
		\end{align}
		for some $k\in \{1,2,3,4,5\}$. For $k=1, (v_1,m_1,m_2)=(1,1,1)$, because $(v_1m_1m_2=1)$. This means that $v_2+v_3=24$, so $v_2v_3=24, v_2-v_3=\sqrt{480}$. This eliminates the possibility of an integer solution.\\
		For $k=2$, Table \ref{t4} shows that no integral solution exists that satisfies these five equations simultaneously.\\
		\begin{table}
			\centering
			\caption{Possible values of $v_1,v_2,v_3,m_1$ and $m_2$ for $k=2: m_1+m_2+v_1v_2+v_1v_3+v_2v_3=100$.}
			\label{t4}
			\scriptsize
			\begin{tabular}{|c|c|c|c|c|c|c|}
				\hline
				$v_1$ & $m_1$ & $m_2$ & $v_2+v_3=25-v_1$ & \begin{tabular}{c}
					$v_2v_3=100-(m_1+m_2+v_1(v_2+v_3))$
				\end{tabular} & $v_2-v_3$\\ \hline
				1&1&16&24&59&$\sqrt{340}$ \\ \hline
				1&16&1&24&59&$\sqrt{340}$  \\ \hline
				16&1&1&9&-46&--   \\ \hline
				1&2&8&24&66&$\sqrt{312}$  \\ \hline
				1&8&2&24&66&$\sqrt{312}$ \\ \hline
				2&1&8&23&45&$\sqrt{349}$   \\ \hline
				2&8&1&23&45&$\sqrt{349}$ \\ \hline
				8&1&2&17&-39&$\sqrt{133}$  \\ \hline
				8&2&2&17&-39&$\sqrt{133}$  \\ \hline
				2&2&4&23&48&$\sqrt{337}$  \\ \hline
				2&4&2&23&48&$\sqrt{337}$  \\ \hline
				4&2&2&21&12&$\sqrt{393}$ \\ \hline
				1&4&4&24&68&$\sqrt{304}$ \\ \hline
				4&1&4&21&11&$\sqrt{30}$  \\ \hline
				4&4&1&21&11&$\sqrt{30}$  \\ \hline
			\end{tabular}
		\end{table}
		For $k=3,$ Table \ref{t5} shows that no feasible integral solution exists for these five equations. \\
		\begin{table}
			\centering
			\caption{Possible values of $v_1,v_2,v_3,m_1$ and $m_2$ for $k=3: m_1+m_2+v_1v_2+v_1v_3+v_2v_3= 150$.}
			\label{t5}
			\scriptsize
			\begin{tabular}{|c|c|c|c|c|c|c|}
				\hline
				$v_1$ & $m_1$ & $m_2$ & $v_2+v_3=25-v_1$ &\begin{tabular}{c}
					$v_2v_3$\\ $=150-(m_1+m_2$\\ $+v_1$$(v_2+v_3))$
				\end{tabular}  & $v_2-v_3$&\begin{tabular}{c}
					$m_1m_2+v_1v_2m_2$\\$+v_1v_3m_1$
				\end{tabular}\\ \hline
				1&1&81&24&44&20&1884  \\ \hline
				1&3&27&24&96&$\sqrt{192}$&--\\ \hline
				1&27&3&24&96&$\sqrt{192}$&--\\ \hline
				3&1&27&22&56&$\sqrt{260}$&--\\ \hline
				3&27&1&22&56&$\sqrt{260}$&--\\ \hline
				1&9&9&24&45&$\sqrt{396}$&--\\ \hline
				9&1&9&16&-4&$\sqrt{240}$&--\\ \hline
				9&9&1&16&-4&$\sqrt{240}$&--\\ \hline
				3&3&9&22&72&14&81\\ \hline
				9&3&3&16&0&16&441 \\ \hline
			\end{tabular}
		\end{table}
		Similarly, Table \ref{t6} lists all possible values of $v_1, m_1,m_2$ for $k=4$. $v_1m_1m_2=256$ gives one possibility of $(v_1,m_1,m_2)=(1,16,16)$. Thus, $v_1m_2+v_2m_2+v_3m_1+v_1v_2v_3+v_1m_1=560$ and $ m_1m_2+v_1v_2m_2+v_1v_3m_1=640$ gives $(v_1,v_2,v_3) \in \{(1,12,12)\}$. Consequently, $(1+z)(1+12z+16z^2)(1+12z+16z^2)$ is a possible component of $G$.
		\begin{table}
			\centering
			\caption{Possible values of $v_1,v_2,v_3,m_1$ and $m_2$ for $k=4: m_1+m_2+v_1v_2+v_1v_3+v_2v_3=200$.}
			\label{t6}
			\scriptsize
			\begin{tabular}{|c|c|c|c|c|c|c|c|}
				\hline
				\textbf{$v_1$} & $m_1$& $m_2$ & \begin{tabular}{c}
				    $v_2+v_3=$\\ $25-v_1$ 
				\end{tabular}&\begin{tabular}{c}
					$v_2v_3=$\\ $200-(m_1+m_2+$\\ $v_1$$(v_2+v_3))$
				\end{tabular}  & $v_2-v_3$& \begin{tabular}{c}
					$m_1m_2$\\$+v_1v_2m_2$ \\ $+v_1v_3m_1$
				\end{tabular} & Independence polynomial \\ \hline
				% 1&1&256&24&-81&30&7165 & Not possible\\ \hline
                1&1&256&24&--&--&-- & Not possible\\ \hline
				1&2&128&24&46&$\sqrt{392}$&-- & Not possible\\ \hline
				1&128&2&24&46&$\sqrt{392}$&--& Not possible\\ \hline
				2&128&1&23&24&$\sqrt{429}$&--&Not possible \\ \hline
				2&1&128&23&24&$\sqrt{429}$&--& Not possible \\ \hline
				1&4&64&24&108&$\pm12$&1432,712& Not possible \\ \hline
				1&64&4&24&108&$\pm12$&1432,712 & Not possible\\ \hline
				4&1&64&21&51&$\sqrt{237}$&-- & Not possible\\ \hline
				4&64&1&21&51&$\sqrt{237}$&-- &Not possible\\ \hline
				1&8&32&24&136&$\sqrt{32}$&-- & Not possible\\ \hline
				1&32&8&24&136&$\sqrt{32}$&-- & Not possible\\ \hline
				8&1&32&17&31&$\sqrt{165}$&-- &Not possible\\ \hline
				8&32&1&17&31&$\sqrt{165}$&-- & Not possible\\ \hline
				1&\textbf{16}&16&24&144&0&640&\begin{tabular}{c}
					$(1+z)(1+12z+16z^2)$\\ $(1+12z+16z^2)$
				\end{tabular}\\ \hline
				16&1&16&9&39&$\sqrt{-75}$&--& Not possible \\ \hline
				16&16&1&9&39&$\sqrt{-75}$&-- & Not possible\\ \hline
				2&2&64&23&88&$\sqrt{177}$&-- &Not possible\\ \hline
				2&64&2&23&88&$\sqrt{177}$&-- &Not possible\\ \hline
				2&4&32&23&118&$\sqrt{57}$&-- & Not possible\\ \hline
				2&32&4&23&118&$\sqrt{57}$&-- & Not possible\\ \hline
				4&2&32&21&82&$\sqrt{113}$&-- &Not possible \\ \hline
				4&32&2&21&82&$\sqrt{113}$&-- & Not possible\\ \hline
				2&8&16&23&130&$\pm3$&704,656 & Not possible\\ \hline
				2&16&8&23&130&$\pm3$&704,656 & Not possible\\ \hline
				8&2&16&17&46&$\sqrt{129}$&-- & Not possible\\ \hline
				8&16&2&17&46&$\sqrt{129}$&-- & Not possible\\ \hline
				16&8&2&9&46&$\sqrt{-103}$&-- & Not possible\\ \hline
				16&2&8&9&46&$\sqrt{-103}$&-- & Not possible\\ \hline
				4&4&16&21&96&$\sqrt{57}$&-- & Not possible\\ \hline
				4&16&4&21&96&$\sqrt{57}$&-- & Not possible\\ \hline
				16&4&4&9&48&$\sqrt{-111}$&--&Not possible \\ \hline
				4&8&8&21&100&$\sqrt{41}$&-- & Not possible\\ \hline
				8&4&8&17&52&$\pm9$&992, 704 &Not possible\\ \hline
				8&8&4&17&52&$\pm9$&992, 704 & Not possible\\ \hline
				
			\end{tabular}
		\end{table}
       For $k=5,$ Table \ref{T6} indicates that no feasible solution exists that satisfies all five equations simultaneously.
        \begin{table}
			\centering
			\caption{Possible values of $v_1,v_2,v_3,m_1$ and $m_2$ for $k=5: m_1+m_2+v_1v_2+v_1v_3+v_2v_3=250$.}
			\label{T6}
			\scriptsize
			\begin{tabular}{|c|c|c|c|c|c|c|}
				\hline
			\textbf{$v_1$} & $m_1$& $m_2$ & $v_2+v_3=25-v_1$ &\begin{tabular}{c}
			$v_2v_3=$\\ $250-(m_1+m_2+$\\ $v_1$$(v_2+v_3))$
		\end{tabular}  & $v_2-v_3$& \begin{tabular}{c}
					$m_1m_2+v_1v_2m_2$ \\ $+v_1v_3m_1$
		\end{tabular} \\ \hline
	1&1&625&24&--&--&-- \\ \hline
    1&5&125&24&96&$\sqrt{80}$&-- \\ \hline
5&5&25&20&120&$\sqrt{-80}$&-- \\ \hline
5&25&5&20&120&$\sqrt{-80}$&-- \\ \hline
  \end{tabular}
		\end{table}
		%	\item 
		
		If $G$ has two connected components $H_1$ and $H_2$ with $v_1$ and $v_2$ vertices respectively, then there are two possible distributions of their independence numbers: either one component has independence number two and the other three, or one has independence number one and the other four.\\
		%\medskip
		\underline{Case I:}
		Assume that the independence numbers of $H_1$ and $H_2$ are two and three, respectively. Suppose further that
		\[I_{H_1}(z) = 1 + v_1 z + m_1 z^2,\quad I_{H_2}(z) = 1 + v_2 z + m_2 z^2 + m_3 z^3\]
		for some natural numbers $m_1, m_2$ and $m_3$. Then
		\begin{align*}
			^kI_G(z)
			&= (1 + v_1 z + m_1 z^2)(1 + v_2 z + m_2 z^2 + m_3 z^3)\\
			&= 1 + (v_1 + v_2)z + (m_1 + m_2 + v_1 v_2)z^2 \\
			&\quad + (a m_2 + b m_1 + m_3)z^3
			+ (a m_3 + m_1 m_2)z^4 + m_1 m_3 z^5.
		\end{align*}
		The admissible values of $v_1, v_2, m_1, m_2,$ and $m_3$ must satisfy
		\begin{align} 
			&	v_1 + v_2 = 25,\label{eq1}\\
			&	m_1 + m_2 + v_1 v_2 = 50k,\label{eq2}\\
			&	v_1 m_2 + v_2m_1 + m_3 = 35k^2,\label{eq3}\\
			&	v_1 m_3 + m_1 m_2 = 10k^3,\label{eq4}\\
			&	m_1 m_3 = k^4,\label{eq5}
		\end{align}
		for some $k \in \{1,2,3,4,5\}$. For $k = 1$, we have $m_1 m_3 = 1$. This implies $v_1 + m_2 = 10$ and this leads to the quadratic equation $v_1^2 - 24 v_1 + 39 = 0$, which yields $v_1 = \sqrt{105}$. Hence, no integer solution arises in this case. For $k=2$, Table \ref{t7} shows that no integral solution exists for these equations.
		\begin{table}
			\centering
			\caption{Possible values of $v_1,v_2,m_1,m_2$ and $m_3$ for $k=2: m_1 + m_2 + v_1 v_2 = 100$.}
			\label{t7}
			%\scriptsize
			\begin{tabular}{|c|c|c|c|c|c|}
				\hline
				$m_1$ & $m_3$ & $m_2=\frac{80-v_1m_3}{m_1}$ &
				$m_1+m_2+v_1v_2$&$v_1 m_2 + v_2m_1 + m_3$ \\ \hline
				1&16&$80-16v_1$& $v_1^2-9v_1+19$&--\\ \hline
				2&8&$40-4v_1$& $v_1^2-21v_1+58$&--\\ \hline
				4&4&$20-v_1$& $v_1^2-24v_1+76$&--\\ \hline

			\end{tabular}
		\end{table}
		For $k=3$, all the possible values $v_1,v_2,m_1,m_2,m_3$ are provided in Table \ref{t8}. In this situation, no possible integral values exist.
		\begin{table}
			\centering
			\caption{Possible values of $v_1,v_2,m_1,m_2$ and $m_3$ for $k=3: m_1 + m_2 + v_1 v_2 = 150$.}
			\label{t8}
			%	\scriptsize
			\begin{tabular}{|c|c|c|c|c|c|}
				\hline
				$m_1$ & $m_3$ & $m_2=\frac{270-v_1m_3}{m_1}$ &
				$m_1+m_2+v_1v_2$&$v_1 m_2 + v_2m_1 + m_3$ \\ \hline
				1&81&$270-81v_1$& $v_1^2+56v_1-121$&--\\ \hline
				3&27&$90-9v_1$& $v_1^2-16v_1+59$&--\\ \hline
				9&9&$30-v_1$& $v_1^2-24v_1+119$&--\\ \hline

			\end{tabular}
		\end{table}
		For $k=4$, all possible values of $v_1,v_2,m_1,m_2,m_3$ are provided in Table \ref{t9}. After simplification we obtain $(v_1,v_2)=(12,13)$. Also, the equations \eqref{eq1}-\eqref{eq5} are satisfied simultaneously. Moreover, we obtain a possible component of the independence polynomial $G$ as $(1+12z+16z^2)(1+13z+28z^2+16z^3)$.
		\begin{table}
			\centering
			\caption{Possible values of $v_1,v_2,m_1,m_2$ and $m_3$ for $k=4: m_1 + m_2 + v_1 v_2 = 200$.}
			\label{t9}
			%	\scriptsize
			\begin{tabular}{|c|c|c|c|c|c|}
				\hline
				$m_1$ & $m_3$ & $m_2=\frac{640-v_1m_3}{m_1}$ &
				$m_1+m_2+v_1v_2$& \begin{tabular}{c}
				   $v_1 m_2 +$\\$ v_2m_1 + m_3$
				\end{tabular}& Independence polynomial \\ \hline
				1&256&$640-256v_1$& $v_1^2+231v_1-442$&--& Not possible\\ \hline
				2&128&$320-64v_1$& $v_1^2+9v_1-140$&--&Not possible\\ \hline
				4&64&$160-16v_1$& $v_1^2-9v_1+36$&--& Not possible\\ \hline
				8&32&$80-4v_1$& $v_1^2-21v_1+112$&--& Not possible\\ \hline
				16&16&$40-v_1$& $v_1^2-24v_1+144$&560&\begin{tabular}{c}
					$(1+12z+16z^2)$\\$(1+13z+28z^2+16z^3)$
				\end{tabular}\\ \hline

			\end{tabular}
		\end{table}
		For $k=5$, all the possible values $v_1,v_2,m_1,m_2,m_3$ are given in Table \ref{T9}. In this situation, possible integral values exist.\\
        \begin{table}
			\centering
			\caption{Possible values of $v_1,v_2,m_1,m_2$ and $m_3$ for $k=5: m_1 + m_2 + v_1 v_2 = 250$.}
			\label{T9}
			%	\scriptsize
			\begin{tabular}{|c|c|c|c|c|c|}
				\hline
				$m_1$ & $m_3$ & $m_2=\frac{1250-v_1m_3}{m_1}$ &
				$m_1+m_2+v_1v_2$&$v_1 m_2 + v_2m_1 + m_3$ \\ \hline
				1&625&$1250-625v_1$& $v_1^2-650v_1-1001$&--\\ \hline
                5&125&$250-25v_1$& $v_1^2-5$&--\\ \hline
                25&25&$50-5v_1$& $v_1^2-20v_1+175$&--\\ \hline
			\end{tabular}
		\end{table}
		\underline{Case II:}
		Assume that the independence numbers of $H_1$ and $H_2$ are one and four, respectively. Let
		\[I_{H_1}(z) = 1 + v_1 z,\quad I_{H_2}(z) = 1 + v_2 z + m_1 z^2 + m_2 z^3+ m_3z^4\]
		for some natural numbers $m_1, m_2$ and $m_3$. Then
		\begin{align*}
			^kI_G(z)
			&= (1 + v_1 z)(1 + v_2 z + m_1 z^2 + m_2 z^3 +m_3z^4)\\
			&= 1 + (v_1 + v_2)z + (m_1 +v_1 v_2)z^2 + (m_2 + v_1m_1)z^3\\
			&\quad + (v_2m_2+m_3)z^4 + v_1m_3 z^5.
		\end{align*}
		The admissible values of $v_1, v_2, m_1, m_2,$ and $m_3$ must satisfy
		\begin{align} 
			&	v_1 + v_2 = 25,\\
			&   m_1 +v_1 v_2= 50k,\\
			&	m_2 + v_1m_1 = 35k^2,\\
			&v_1m_2+m_3 = 10k^3,\\
			&	v_1m_3= k^4,
		\end{align}
		for some $k \in \{1,2,3,4,5\}$. When $k = 1$, we get $v_1m_3 = 1$, which gives $v_2 = 24$, $m_1 = 26$, and $m_2 = 9$. Also, the condition $v_1m_2 +m_3 = 10k^3$ is satisfied. As a result, the possible independence polynomial is $(1+z)(1+24z+26z^2+9z^3+z^4)$. For $k = 2$, Table \ref{t10} shows that the only integral solution is obtained when $(v_1, v_2) = (1, 24)$ and the independence polynomial is $(1+z)(1+24z+76z^2+64z^3+16z^4)$	
		
		\begin{table}
			\centering
			\caption{Possible values of $v_1,v_2,m_1,m_2$ and $m_3$ for $k=2: m_2 + v_1m_1 =140$.}
			\label{t10}
			%	\scriptsize
			\begin{tabular}{|c|c|c|c|c|c|}
				\hline
				$v_1$ & $m_3$ & $m_2=\frac{80-m_3}{v_1}$ &
				$m_1= 100-v_1v_2$&$m_2+v_1m_1$& Independence polynomial \\ \hline
				1&16&64& 76&140&\begin{tabular}{c}
					$(1+z)(1+24z+76z^2$\\$+64z^3+16z^4)$	\end{tabular} \\ \hline
				2&8&36& 54&164& Not possible\\ \hline
				%8&2&& -36&299\\ \hline
				4&4&19&16&67& Not possible\\ \hline
				
			\end{tabular}
		\end{table}
		For $k=3$, Table \ref{t11}, shows that the only integral solution is obtained when $(v_1, v_2) = (1,24)$ with independence polynomial is $(1+z)(1+24z+126z^2+189z^3+81z^4)$.
		%		 	we have $m_1 +v_1 v_2= 150, m_2 + v_1m_1 = 315, v_1m_2+m_3 = 270$ and $	v_1m_3= 81$. 
		\begin{table}
			\centering
			\caption{Possible values of $v_1,v_2,m_1,m_2$ and $m_3$ for $k=3: m_2 + v_1m_1 =315$.}
			\label{t11}
			%\scriptsize
			\begin{tabular}{|c|c|c|c|c|c|}
				\hline
				$v_1$ & $m_3$ & $m_2=\frac{270-m_3}{v_1}$ &
				$m_1= 150-v_1v_2$&$m_2+v_1m_1$& Independence polynomial \\ \hline
				1&81&189&126&315&\begin{tabular}{c}
					$(1+z)(1+24z+126z^2$\\$+189z^3+81z^4)$
				\end{tabular} \\ \hline
				3&27&81&84&333 & Not possible\\ \hline
				9&9&29&6&83&Not possible \\ \hline
				
			\end{tabular}
		\end{table}
		For $k=4$, Table \ref{t12} shows that the only integral solution exist when $(v_1,v_2)=(1,24)$ with $(1+z)(1+24z+176z^2+384z^3+256z^4)$ as independence polynomial.
		%	\end{enumerate}
	\begin{table}
		\centering
		\caption{Possible values of $v_1,v_2,m_1,m_2$ and $m_3$ for $k=4: m_2 + v_1m_1 =560$.}
		\label{t12}
		%	\scriptsize
		\begin{tabular}{|c|c|c|c|c|c|}
			\hline
			$v_1$ & $m_3$ & $m_2=\frac{640-m_3}{v_1}$ &
			$m_1= 200-v_1v_2$&$m_2+v_1m_1$& Independence polynomial \\ \hline
			1&256&384&176&560 &\begin{tabular}{c}
				$(1+z)(1+24z+176z^2$\\$+384z^3+256z^4)$
			\end{tabular} \\ \hline
			2&128&256&154&564& Not possible \\ \hline
			8&32&76&64&588&Not possible\\ \hline
			16&16&39&56&935&Not possible \\ \hline
			
		\end{tabular}
	\end{table}
    Table \ref{T12} shows that no solution exists for $k=5.$
\begin{table}
		\centering
		\caption{Possible values of $v_1,v_2,m_1,m_2$ and $m_3$ for $k=5: m_2 + v_1m_1 =875$.}
		\label{T12}
		%	\scriptsize
		\begin{tabular}{|c|c|c|c|c|}
			\hline
			$v_1$ & $m_3$ & $m_2=\frac{1250-m_3}{v_1}$ &
			$m_1= 250-v_1v_2$&$m_2+v_1m_1$ \\ \hline
			1&625&625&226&851 \\ \hline
			5&125&225&150&975\\ \hline      
    \end{tabular}
	\end{table}
\end{proof}
\textbf{Note:}	For $k \in \{1,2,3,4,5\}$ we obtain non-integer solutions of $v_1,v_2,m_1,m_2,m_3$ which we have deliberately not included in Tables (\ref{t7}-\ref{T12}).

\begin{remark}
    In admissible decompositions, the factor $(1+z)$ appears at most once, meaning that $-1$ is a simple root of $G$. Thus, the multiplicity of the root $-1$ does not occur.
\end{remark}

\section{graph construction}\label{sec 4}
We now find some of the disconnected graphs whose independence polynomials are mentioned in Proposition \ref{prop 1}. Basically, we find non-isomorphic connected graphs with the following independence polynomials: $(1+12z+16z^2); (1+13z+28z^2+16z^3); (1+24z+26z^2+9z^3+z^4);(1+24z+76z^2+64z^3+16z^4); (1+24z+126z^2+189z^3+81z^4); (1+24z+176z^2+384z^3+256z^4).$
\begin{remark}
    We are listing only some of the non-isomorphic connected graphs for illustration purposes (the graphs have been drawn using the SageMath tool; we have included only the image output, and the code has not been included).
\end{remark}
Following \cite{khetawat2025circles}, our next result determines some of the non-isomorphic connected graphs with the following independence polynomial.
\begin{lemma}
    Let $Q_1=(1+12z+16z^2), Q_2= (1+13z+28z^2+16z^3), Q_3= (1+24z+26z^2+9z^3+z^4), Q_4=(1+24z+76z^2+64z^3+16z^4), Q_5=(1+24z+126z^2+189z^3+81z^4)$ and $Q_6=(1+24z+176z^2+384z^3+256z^4)$.  Here, $Q_i$ denotes the component of the graph $G$.
\end{lemma}
\begin{enumerate}
    \item  Let \( H_1 \) be a graph with 12 vertices having independence polynomial \( I_{H_1}(z) = 1 + 12z + 16z^2 \). Since this polynomial does not include \( z^3 \) term, it indicates that the complement of \( H_1 \) is \( K_3 \)-free. Consequently, the largest independent set in \( H_1 \) can have a maximum of two vertices. Some possible realization of this graph is illustrated in Figure~\ref{fig 1}.
    \item  Let \( H_2 \) be a graph with 13 vertices, represented by the independence polynomial \( I_{H_2}(z) = 1 + 13z + 28z^2 + 16z^3 \). The appearance of a non-zero coefficient for \( z^3 \) indicates that the complement of \( H_2 \) contains sixteen triangles, yet it is still \( K_4 \) free. A concrete example of \( H_2 \) is illustrated in Figure~\ref{fig 2}.
    \item Consider the graph \( H_3 \) such that $I_{H_3}(z) = 1 + 24z + 26z^2 + 9z^3 + z^4.$ Then the complement of $H_3$ has 24 vertices, 26 edges, and 9 triangles. The absence of  \( z^5 \) term indicates that the complement of \( H_3 \) is \( K_5 \) free. While these graphs share the same triangle-counting polynomial, they differ in their local configurations and connectivity patterns. 
    % Figure \ref{} shows six representative examples.
    \item 
    Let \( H_4 \) be a graph with 24 vertices, characterized by its independence polynomial \( I_{H_4}(z) = 1 + 24z + 76z^2 + 64z^3 + 16z^4 \). The complement of \( H_4 \) has 76 edges, 64 triangles, and 16 subgraphs that are isomorphic to \( K_4 \), while it does not contain any \( K_5 \) subgraphs. Five pairwise non-isomorphic realizations of the complement of \( H_4 \) are presented in Figure~\ref{fig 3}.
    \item  Let \(H_5\) be a graph on 24 vertices whose independence polynomial is given by \( I_{H_5}(z) = 1 + 24z + 126z^2 + 189z^3 + 81z^4 \). The complement of \(H_5\) contains 126 edges, 189 triangles, and 81 copies of \(K_4\), and it has no subgraphs isomorphic to \(K_5\). Two mutually non-isomorphic realizations of the complement of \(H_5\) are shown in Figure~\ref{fig 4}.
    \item Let \(H_6\) be a graph on 24 vertices with independence polynomial \( I_{H_6}(z) = 1 + 24z + 176z^2 + 384z^3 + 256z^4 \). The same diagram as in \eqref{fig 1} can be used to represent this graph, since its independence polynomial is the square of the independence polynomial of the graph \(H_1\).
\end{enumerate}
\begin{figure}
    \centering
    \includegraphics[width=0.5\linewidth]{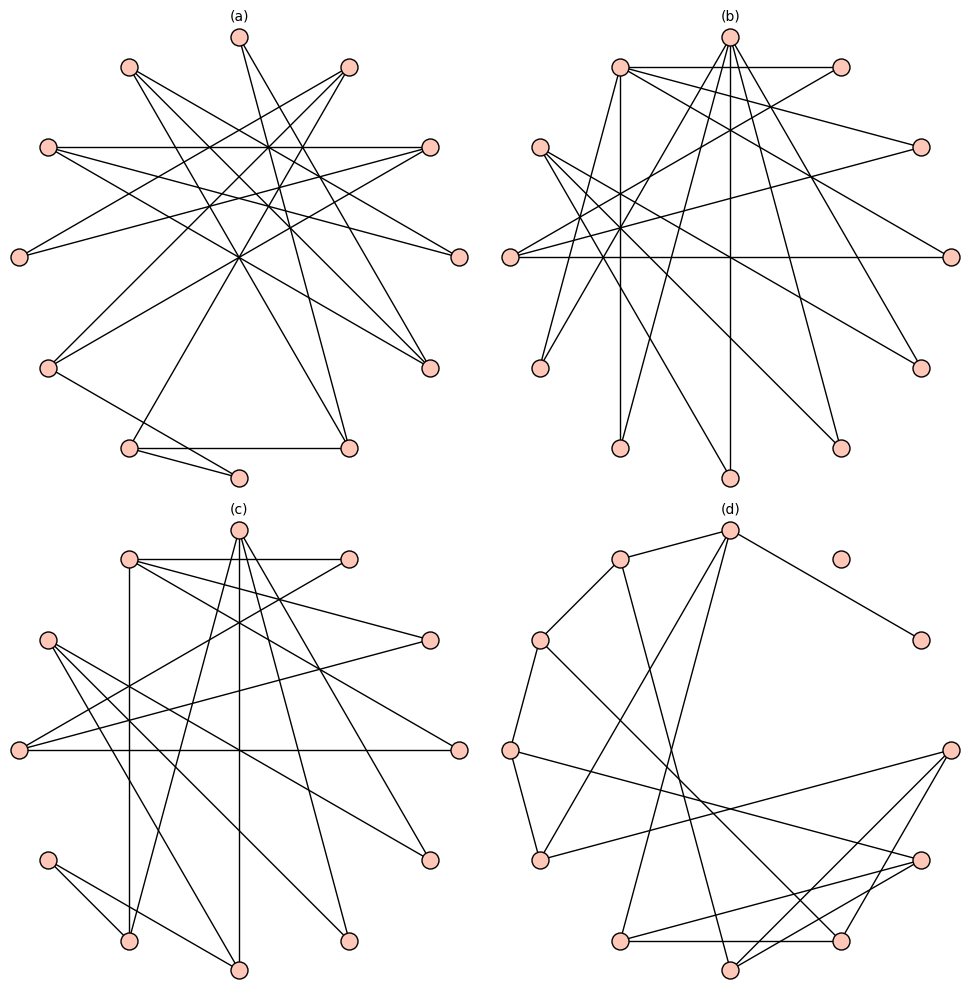}
    \caption{Some possible non-isomorphic connected graphs of the complement of $H_1$ where $I_{H_1}(z)=1+12z+16z^2$.}
    \label{fig 1}
\end{figure}

\begin{figure}
    \centering
    \includegraphics[width=0.6\linewidth]{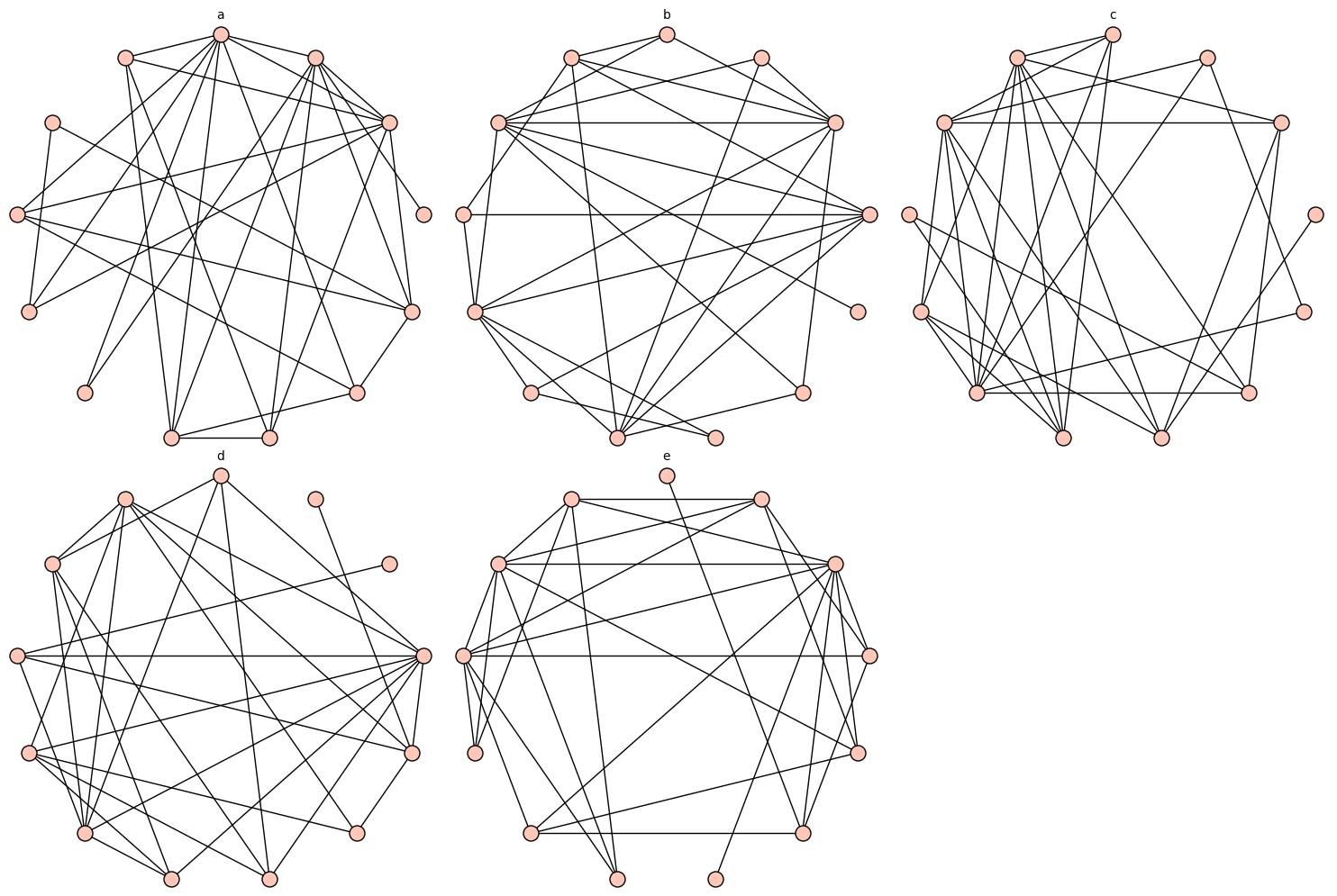}
    \caption{Some possible non-isomorphic connected graphs of the complement of $H_2$ where $I_{H_2}(z)=1+13z+28z^2+16z^3$.}
    \label{fig 2}
\end{figure}
\begin{figure}
    \centering
    \includegraphics[width=0.7\linewidth]{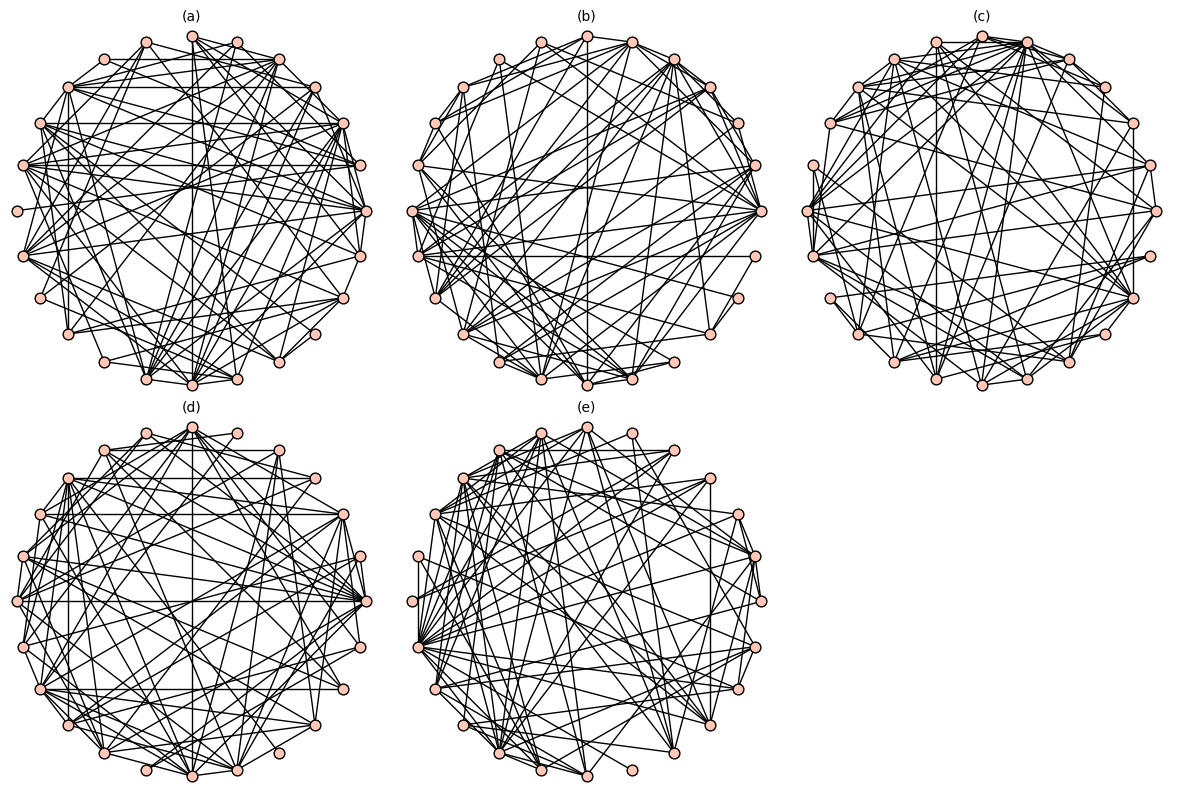}
    \caption{Some possible non-isomorphic connected graphs of the complement of $H_4$ where $I_{H_4}(z)=1+24z+76z^2+64z^3+16z^4$.}
    \label{fig 3}
\end{figure}
\begin{figure}[htbp]
    \centering

    \begin{minipage}{0.25\textwidth}
        \centering
        \includegraphics[width=\textwidth]{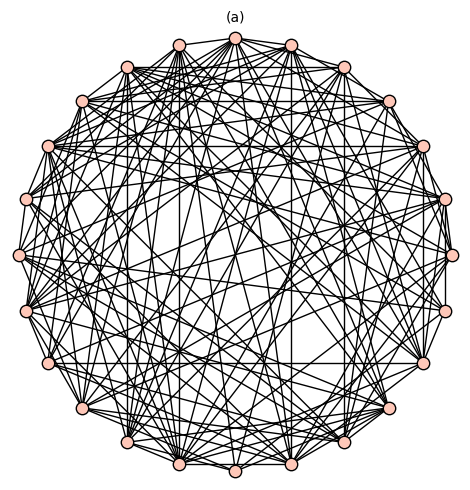}
    \end{minipage}
    \hspace{0.02\textwidth}
    \begin{minipage}{0.25\textwidth}
        \centering
        \includegraphics[width=\textwidth]{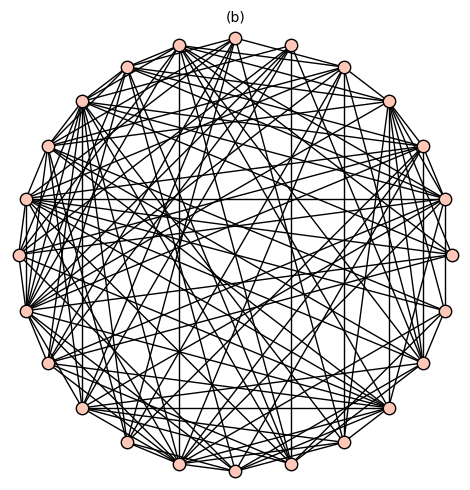}
    \end{minipage}

    \caption{Some possible non-isomorphic connected graphs of the complement of $H_5$ where $I_{H_5}(z)=1+24z+126z^2+189z^3+81z^4$.}
    \label{fig 4}
\end{figure}
% \newpage
	% \bibliographystyle{abbrv}
	% \bibliography{ref}

\end{document}